\documentclass[12pt]{article}

\usepackage[english]{babel}
\usepackage[margin=1in]{geometry}

\usepackage{amsmath}
\usepackage{amsthm}
\usepackage{amssymb}
\usepackage{hyperref}
\usepackage{microtype}

\usepackage{biblatex}
\usepackage{amsmath}
\usepackage{graphicx}

\DeclareMathOperator{\fix}{Fix}

\DeclareMathOperator{\expand}{Exp}
\DeclareMathOperator{\diam}{diam}

\theoremstyle{plain}

\newtheorem{lemma}{Lemma}
 
\newtheorem{theorem}{Theorem}
\newtheorem{cor}{Corollary}

\numberwithin{theorem}{section}
\numberwithin{cor}{section}
\numberwithin{lemma}{section}

\theoremstyle{definition}
\newtheorem{definition}{Definition}

\theoremstyle{remark}
 
\newtheorem{remark}{Remark}
\newtheorem*{acknowledgements}{Acknowledgments}

\title{Finite Data Rigidity for One-Dimensional Expanding Maps}
\author{Thomas Aloysius O'Hare}

\begin{document}
\maketitle
\begin{abstract}
    \noindent Let $f,g$ be $C^2$ expanding maps on the circle which are topologically conjugate. We assume that the derivatives of $f$ and $g$ at corresponding periodic points coincide for some large period $N$. We show that $f$ and $g$ are ``approximately smoothly conjugate." Namely, we construct a $C^2$ conjugacy $h_N$ such that $h_N$ is exponentially close to $h$ in the $C^0$ topology, and $f_N:=h_N^{-1}gh_N$ is exponentially close to $f$ in the $C^1$ topology. Our main tool is a uniform effective version of Bowen's equidistribution of weighted periodic orbits to the equilibrium state.

\end{abstract}
\section{Introduction}
A $C^1$ map $f:S^1\rightarrow S^1$ is called expanding if $\min_{x\in S^1}|f'(x)|\geq\lambda_f>1$. We call $\lambda_f$ the minimum expansion rate. Let $\expand^r(S^1)$ ($r\geq 1$) be the subspace of all $C^r$ uniformly expanding maps, and for $\gamma>1$ let $\expand^r_\gamma(S^1)$ be the space of all $C^r$ expanding maps whose minimum expansion rate is greater than or equal to $\gamma$.
Given any continuous map
$f:S^1\rightarrow S^1$, recall that the degree of $f$ is defined to be value $F(x+1)-F(x)$, where $F:\mathbb{R}\rightarrow\mathbb{R}$ is any lift of $f$. In addition to being well-defined independent of the choice of lift and the point $x\in S^1$, it was proved by Shub in \cite{Shub_1969} that the degree is a complete topological conjugacy invariant for expanding maps on the circle:
\begin{theorem}
Let $f,g:S^1\rightarrow S^1$ be continuous expanding maps. Then there exists a homeomorphism $h:S^1\rightarrow S^1$ such that $h\circ f=g\circ h$ if and only if $\deg(f)=\deg(g)$.
\end{theorem}
It is easy to check that $h$ is H\"older continuous, but there is an obstruction to $h$ having higher regularity at the periodic orbits of $f$. By formally differentiating the conjugacy equation, it is clear that $h$ will not be differentiable if $(f^n)'(x)\neq (g^n)'(h(x))$ for at least one periodic point $x\in\fix(f^n)$. As the next theorem shows, satisfying this obstruction at all periodic points is sufficient to conclude differentiability of $h$:

\begin{theorem}
Suppose $f,g:S^1\rightarrow S^1$ are $C^{1+\alpha}$ ($\alpha>0$) expanding maps of the same degree, and fix a conjugacy $h$ such that $h\circ f=g\circ h$. Then the map $h$ is $C^{1+\alpha}$ if and only if for every point $p\in \fix(f^n)$, $n\in\mathbb{N}$, we have $(f^n)'(p)=(g^n)'(h(p))$.
\end{theorem}
While not directly stated in this form, a proof of Theorem 1.2 can be found in \cite{delaLlave}.
A closely related result is the theorem of Shub and Sullivan \cite{shub_sullivan_1985}, which proves that $C^r$ ($r\geq 2$) expanding maps of the circle which are conjugated by an absolutely continuous homeomorphism $h_1$ are in fact conjugated by a $C^r$ diffeomorphism $h_2,$ but it may be that $h_1\neq h_2$.
Martens and de Melo prove a more general version of Theorem 1.2 (\cite{Martens-Marco} Corollary 2.9) applying to all $C^r$ Markov maps in one-dimension. In particular, the main theorem of \cite{Martens-Marco} applies to unimodal maps (with critical points) without wild and solenoidal attractors and establishes periodic data rigidity for these systems. The picture for general unimodal maps is more complicated: Moreira and Smania proved \cite{Moreira-Smania} that unimodal maps with Cantor set attractor, such as Feigenbaum maps and Fibonacci maps with high order at the critical point, are always absolutely continuously conjugated, but not necessarily smoothly. Thus Shub and Sullivan's theorem does not carry over to these maps. It would be interesting to study finite data rigidity (Theorem 1.3 below) for unimodal and Markov maps as well.

The goal of the present paper is to relax the conditions of Theorem 1.2 to hold at only finitely many periodic orbits and then see how close $f$ and $g$ are to being smoothly conjugated:
\begin{theorem}
Let $g:S^1\rightarrow S^1$ be a $C^2$-expanding map.Then for every $C_0>0$ and every $\gamma>1$, there exist constants $K>0$ and $0<\lambda<1$ such that the following holds: If $f\in\expand^2_\gamma(S^1)$ is conjugated to $g$ by a homeomorphism $h$ ($h\circ f=g\circ h$), $d_{C^2}(f,g)< C_0$, and if there exists $N\in\mathbb{N}$ such that $(f^n)'(p)=(g^n)'(h(p))$ for every $p\in\fix(f^N)$, then there exists a diffeomorphism ${h}_N\in C^2(S^1)$ such that $d_{C^0}(h,{h}_N)\leq K\lambda^N$. Moreover, for every $0<\lambda^{1/2}<\lambda_0<1$ there exists a constant $K'>0$, such that if we let 
$\overline{f}_N={h}_N^{-1}\circ g\circ{h}_N$, then
$d_{C^1}(f,{f}_N)\leq K'\lambda_0^N$.

\end{theorem}
A key step in the proof of Theorem 1.3 is to prove an effective version of Bowen's equidistribution theorem (see Theorem 2.1 below), which allows us to estimate the difference between $h$ and ${h}_N$ in terms of the periodic orbits of order $N$. The convergence rate $\lambda$ comes from the effective equidistribution rate and depends on the degree of the expanding maps and $\gamma$. The proof of effective equidistribution is postponed until Section 3 and relies on the technique of Birkhoff cones for subshifts of finite type, which we recall in the appendix. This technique is well-known in the case of expanding maps; see for instance \cite{Baladi}.

\begin{acknowledgements}
    The author would like to express his sincerest thanks to Andrey Gogolev for suggesting the problem, and for his patience and invaluable guidance. The author would also like to thank James Marshall Reber for many useful discussions and for his help in proving Theorem 3.1. Thanks to Daniel Smania for providing valuable comments on the draft of this paper.
\end{acknowledgements}

\section{Finite Data Rigidity}
The goal of this section is to generalize Theorem 1.2 to allow for the derivatives of $f$ and $g$ to agree only at finitely many periodic points. Of course, $f$ and $g$ will not be $C^1$ conjugate, but we can find a map $C^1$ close to $f$ which is $C^1$ conjugate to $g$. Moreover, this new map converges exponentially to $f$ as the number of periodic points that the derivatives agree on increases.
\begin{definition}
For a function $f\in C^k(S^1)$, $k\in\mathbb{N}$, let $|f|_{C^k}=\sup|D^kf|$ denote the $C^k$ seminorm. Let 
$||f||_{\infty}$ denote the supremum norm of $f$ and let
$||f||_{C^k}=||f||_{\infty}+\sum_{i=1}^k|f|_{C^i}$ denote the $C^k$ norm. 
If $f:S^1\rightarrow S^1$ is Lipschitz continuous, define
$$|f|_{Lip}=\sup_{x\neq y}\frac{|f(x)-f(y)|}{|x-y|} $$
to be the Lipschitz seminorm, and define $||f||_{Lip}=||f||_{\infty}+|f|_{Lip}$ to be the Lipschitz norm.
\end{definition}
Let us recall the following theorem of Bowen \cite{Bowen_1974}:
\begin{theorem}[Bowen]
Let $(X,d)$ be a compact metric space, $f:X\rightarrow X$ an expansive homeomorphism with the specification property, and $\psi\in C^f(X)$. Then there exists a unique equilibrium state $\mu_{\psi}\in\mathcal{M}(f)$ given by
$$\mu_{\psi}=\lim_{n\rightarrow\infty}\frac{1}{Z_n(\psi)}\sum_{x\in\fix(f^n)}\exp(S_n\psi(x))\delta_x,$$
where
$$
Z_n(\psi)=\sum_{x\in\fix(f^n)}\exp(S_n\psi(x))$$
is a normalization constant and the limit converges in the weak$^*$-topology.
\end{theorem}

\begin{proof}[Proof of Theorem 1.3]
Let $\mathcal{W}:=\{f\in\expand^2_\gamma(S^1)\hspace{1mm} |\hspace{1mm} d_{C^2}(f,g)<C_0\}.$
Let $\psi_f=-\log(|f'|)$ and $\psi_g=-\log(|g'|)$.
We begin by recalling the main ideas of the proof of Theorem 1.2 that we need. It is well known that for a $C^{1+\alpha}$ expanding map $f$ on a compact manifold $M$, there exists a unique invariant probability measure $\mu_f$ which is absolutely continuous with respect to Lebesgue measure, whose density $\rho_f(x)$ is $C^\alpha$ and strictly positive. See \cite{Baladi} for full details. Moreover, $\mu_f$ is the unique equilibrium state corresponding to the potential $\psi_f$. The assumption on periodic data allows us to use Bowen's theorem to conclude that $\mu_f=h^*\mu_g$, where $h^*\mu_g$ denotes the pullback measure by $h$. Defining the functions 
$$I_f(x)=\int_0^x\rho_f(y)dy,\hspace{1mm} I_g(x)=\int_0^x\rho_g(y)dy,$$
and integrating $\mu_f=h^*\mu_g$ from $0$ to $x$, we find that $I_f(x)=I_g(h(x)),$ or $h=I_g^{-1}\circ I_f\in C^{1+\alpha}$. Observe that the latter expression $I_g^{-1}\circ I_f$ is a well-defined $C^{1+\alpha}$ function without any hypotheses on the periodic data. For this reason, we define $h_N:=I_g^{-1}\circ I_f$. In what follows we will need the following uniform bound on the densities:
\begin{lemma}
    Let $\mathcal{W}$ be a bounded open set in $\expand^2_\gamma(S^1)$, $\gamma>1$. Then there exists $C>1$ such that for all $f\in\mathcal{W}$ and every $x\in S^1$, $C^{-1}\leq \rho_f(x)\leq C$, and $|\rho_f|_{C^\alpha}<C$.
\end{lemma}
\begin{proof}
    For $0<\alpha<1$, let $\Phi: \expand^{1+\alpha}(S^1)\rightarrow C^\alpha(S^1)$ be the map sending an expanding map to it's invariant density: $\Phi(f)=\rho_f$. By \cite{Baladi} Theorem 2.10, this map is continuous. Let $U_n=\Phi^{-1}(\{\rho\in C^\alpha(S^1) | \frac{1}{n}<\rho< n, |\rho|_{C^\alpha}<C \})$. Then $\{U_n\}$ is an increasing sequence of open sets that covers  $\expand^{1+\alpha}$. Since $\mathcal{W}$ is a $C^2$ open and bounded set, it is relatively compact in the $C^{1+\alpha}$ topology by the Arzela-Ascoli theorem. Hence the closure $\overline{\mathcal{W}}$ is compact in $\expand^{1+\alpha}$, and hence must be entirely contained in one of the $U_n$. This is precisely the desired conclusion.
\end{proof}
\begin{remark}
      One can prove uniform bounds on the densities without using continuity of the map $\Phi$ by instead carefully going through the arguments of Sacksteder's proof \cite{Sacksteder} of the existence of invariant densities.
\end{remark}
\noindent Let $$\mu_f^n=\frac{1}{Z_n(\psi_f)}\sum_{y\in\fix(f^n)}\exp(S_n\psi_f(y))\delta_y$$
be the discrete measures occurring in Bowen's theorem. Observe that since  $(f^n)'(p)=(g^n)'(h(p))$ for every $p\in\fix(f^N)$, we have that $\mu_f^N=h^*\mu_g^N$:
$$
h^*\left(\frac{1}{Z_N(\psi_g)}\sum_{x\in\fix(g^N)}\exp(S_N\psi_g(x))\delta_x\right)=
\frac{1}{Z_N(\psi_g)}\sum_{x\in\fix(g^N)}\exp(S_N\psi_g(h(h^{-1}x)))\delta_{h^{-1}(x)}$$
$$
=\frac{1}{Z_N(\psi_f)}\sum_{y\in\fix(f^N)}\exp(S_N\psi_f(y))\delta_y=\mu_f^N.
$$ 
We will need the following theorem in order to estimate $d_{C^0}(h,{h}_N)$.

\begin{theorem}
Let $\mathcal{W}$ be as in the proof of Theorem 1.3. Let $\psi:S^1\rightarrow\mathbb{C}$ be a Lipschitz continuous potential with unique equilibrium state $\mu_\psi$, and let $\mu_\psi^n$ be weighted discrete measures supported on $\fix(f^n)$. Then there exist constants $C'>0$ and $0<\tau<1$, depending only on $\mathcal{W}$, such that for every Lipschitz function $\phi:S^1\rightarrow \mathbb{C}$, we have
$$\bigg| \int\phi d\mu_f-\int\phi d\mu_f^N\bigg|\leq C'||\phi||_{Lip}\tau^N. $$
\end{theorem}
\noindent
We defer the proof of Theorem 2.2 to section 3. See \cite{Kadyrov} theorem 1.5 for a more general result for the measure of maximal entropy of a subshift of finite type, and \cite{Ruhr} for a version applying to equilibrium states of countable state shifts. We now calculate 
$$|{h}_N(x)-h(x)|=|I_g^{-1}\circ I_f(x)-h(x)|=
|I_g^{-1}\circ I_f(x)-I_g^{-1}\circ I_g\circ h(x)|\leq
C |I_f(x)-I_g(h(x))|$$
$$
=C\bigg| \int_0^xd\mu_f-\int_0^{h(x)}d\mu_g\bigg|
\leq C\bigg|\int_0^xd\mu_f-\int_0^{h(x)}d\mu_g^N\bigg|
+C\bigg|\int_0^{h(x)}d\mu_g^N-\int_0^{h(x)}d\mu_g\bigg|$$
$$
=C\bigg|\int_0^xd\mu_f-\int_0^xd(h^*\mu_g^N)\bigg|
+C\bigg|\int_0^{h(x)}d\mu_g^N-\int_0^{h(x)}d\mu_g\bigg|$$
$$
=C\bigg|\int_0^xd\mu_f-\int_0^xd\mu_f^N\bigg|
+C\bigg|\int_0^{h(x)}d\mu_g^N-\int_0^{h(x)}d\mu_g\bigg|,$$
where $\sup|(I_g^{-1})'|=\sup|\rho_g^{-1}|\leq C$.
By symmetry, it suffices to show that 
$$\bigg|\int_0^xd\mu_f-\int_0^xd\mu_f^N\bigg|\leq K\lambda^N,$$
uniformly in $x\in S^1$. We split the proof into 3 cases:

\noindent
\textbf{Case 1} Assume that $2\tau^{N/2}\leq x\leq 1-2\tau^{N/2}$:
Fix $2\tau^{N/2}\leq x\leq 1-2\tau^{N/2}$ and define a family of continuous functions as follows. 
\[ \phi_x^s(y)=
    \begin{cases}
        \tau^{-N/2}(y+s) & -s\leq y\leq \tau^{N/2}-s \\
        1 & \tau^{N/2}-s\leq y\leq s-(\tau^{N/2}-x) \\
        -\tau^{-N/2}(y-s)+\tau^{-N/2}x & s-(\tau^{N/2}-x)\leq y\leq x+s \\
        0 & x+s\leq y\leq 1-s; \\
    \end{cases}
\]
where $\tau$ is as in Theorem 2.2.
Next, we estimate
$$
\bigg|\int_0^xd\mu_f-\int_0^xd\mu_f^N\bigg|\leq
\bigg|\int_0^xd\mu_f-\int\phi_x^sd\mu_f\bigg|
+\bigg|\int\phi_x^sd\mu_f-\int\phi_x^sd\mu_f^N\bigg|
+\bigg|\int\phi_x^sd\mu_f^N-\int_0^xd\mu_f^N\bigg|$$
Let us consider each term on the right separately. The first term can be rewritten as 
$$
\bigg|\int_0^xd\mu_f-\int\phi_x^sd\mu_f\bigg|
=\bigg|\int(\chi_{[0,x]}(y)-\phi_x^s(y))\rho(y)dy\bigg|
$$
Observe $|\chi_{[0,x]}(y)-\phi_x^s(y)|=0$ except on a set of measure less than $2\tau^{N/2}$ (depending on $s$), and is otherwise bounded above by $1$. Since in addition $\rho_f$ is continuous and hence bounded, we find that 
$$\bigg|\int(\chi_{[0,x]}(y)-\phi_x^s(y))\rho(y)dy\bigg|\leq C\tau^{N/2}.$$
For the second term, we observe that $\phi_x^s$ is Lipschitz for every $x$ and $s$ with $||\phi_x^s||_{Lip}=1+\tau^{-N/2}$. We can therefore apply Theorem 2.2 to find
$$\bigg|\int\phi_x^sd\mu_f-\int\phi_x^sd\mu_f^N\bigg|\leq
C(1+\tau^{-N/2})\tau^N=C(\tau^N+\tau^{N/2})\leq C\tau^{N/2}$$
We claim that the final term is zero for some choice of $s$.
Define
$$\Phi(s)=\int\phi_x^sd\mu_f^N-\int_0^xd\mu_f^N$$
Then $\Phi(0)<0$ since $\phi_x^0\leq \chi_{[0,x]}$ and this inequality is strict on $[0,\tau^{N/2})$. Similarly, $\Phi(\tau^{N/2})>0$. We claim that $\Phi(s)$ is a continuous function of $s$. Indeed, fixing $\varepsilon>0$, we can find a $\delta>0$ such that if $|s_1-s_2|<\delta$, then
$||\phi_x^{s_1}-\phi_x^{s_2}||_{\infty}<\varepsilon$. Then
$$
|\Phi(s_1)-\Phi(s_2)|=\bigg|\int(\phi_x^{s_1}-\phi_x^{s_2})d\mu_f^N\bigg|<\varepsilon\mu_f^N(0,1)=\varepsilon.
$$
By the intermediate value theorem we can therefore choose some $0<s<\tau^{N/2}$ so that $\Phi(s)=0$.

\noindent
\textbf{Case 2} Assume that $x>1-2\tau^{N/2}$: The proof is nearly identical as in Case 1 but we use a slightly different family of Lipschitz functions:
\[ \phi_x^s(y)=
    \begin{cases}
        \tau^{-N/2}(y+s) & -s\leq y\leq \tau^{N/2}-s \\
        1 & \tau^{N/2}-s\leq y\leq s-(\tau^{N/2}-x) \\
        -\tau^{-N/2}(y-s)+\tau^{-N/2}x & s-(\tau^{N/2}-x)\leq y\leq x+s \\
        0 & x+s\leq y\leq 1-s; \\
    \end{cases}
\]
for $0\leq s\leq \frac{1+x}{2}$, and 
\[ \phi_x^s(y)=
    \begin{cases}
      1 & \tau^{N/2}-s\leq y\leq s-(\tau^{N/2}-x) \\
       -\tau^{-N/2}(y-s)+\tau^{-N/2}x & s-(\tau^{N/2}-x)\leq y\leq \frac{1+x}{2} \\
        \tau^{-N/2}(y-1+s) & \frac{1+x}{2}\leq y\leq 1-s; \\
    \end{cases}
\]
for $\frac{1+x}{2}\leq s\leq \tau^{N/2}$. The remainder of the proof is identical to Case 1.

\noindent
\textbf{Case 3} Assume that $x\leq 2\tau^{N/2}$: The proof is again identical to Case 1 but with the following family:
\[ \phi_x^s(y)=
    \begin{cases}
      \tau^{-N/2}(y+s) & -s\leq y\leq \frac{x}{2} \\
       -\tau^{-N/2}(y-s)+\tau^{-N/2}x & \frac{x}{2}\leq y\leq s \\
        0 & s\leq y\leq 1-s; \\
    \end{cases}
\]
for $0\leq s\leq \tau^{N/2}-\frac{x}{2}$, and
\[ \phi_x^s(y)=
    \begin{cases}
      \tau^{-N/2}(y+s) & -s\leq y\leq \tau^{N/2}-s \\
      1 & \tau^{N/2}-s\leq y\leq s+\tau^Nx-\tau^{-N/2}\\
       -\tau^{-N/2}(y-s)+\tau^{-N/2}x & s+\tau^Nx-\tau^{-N/2}\leq y\leq s \\
        0 & s\leq y\leq 1-s; \\
    \end{cases}
\]
for $\tau^{N/2}-\frac{x}{2}\leq s\leq \tau^{N/2}$. Observe that for every $0\leq s\leq \tau^{N/2}$, $||\phi_x^s||_{\infty}\leq 1$, so that $||\phi_x^s||_{Lip}\leq 1+ \tau^{-N/2}$, so the remainder of the argument in Case 1 carries over verbatim.

In each case, we combine our bounds on these three terms to conclude that for every $x\in S^1$, $|{h}_N(x)-h(x)|\leq K\tau^{N/2}$. Setting $\lambda=\tau^{1/2}$ yields the first stated conclusion.
As a first consequence we obtain a bound on the $C_0$ distance between $f$ and $f_N$:
$$
|f(x)-f_N(x)|=|h^{-1}(g((h(x))-{h}_N^{-1}(g({h}_N(x))|\leq$$
$$
|h^{-1}(g((h(x))-{h}_N^{-1}(g(h(x))|+
|{h}_N^{-1}(g((h(x))-{h}_N^{-1}(g({h}_N(x))|\leq$$
$$
d_{C^0}(h^{-1},{h}_N^{-1})+Lip({h}_N^{-1}\circ g)d_{C^0}(h,{h}_N)\leq K(1+Lip({h}_N^{-1}\circ g))\lambda^N.
$$
Note that $Lip({h}_N^{-1}\circ g)=\sup|D({h}_N^{-1}\circ g)|=\sup(|D(I_f^{-1}\circ I_g\circ g)|\leq\frac{\max \rho_g}{\min \rho_f}\max|g'|$, which is uniformly bounded in $\mathcal{W}$.
\begin{lemma}
Fix $M>0$ and $0<\alpha\leq 1$. Let $\phi:S^1\rightarrow S^1$ be a $C^{1+\alpha}$ function with $|\phi'|_{C^\alpha}<M$, and let $\varepsilon,\delta>0$ be such that 

$$\sup_{|x-y|>\delta}\frac{|\phi(x)-\phi(y)|}{|x-y|}<\varepsilon,$$
then $|\phi|_{C^1}<\frac{M}{\alpha+1}\delta^\alpha+\epsilon$.
\end{lemma}
\begin{proof}
Suppose first that $\phi(0)=0$ and $\phi'(0)=\sup |\phi'|=:\varepsilon'$, and take $|x|>\delta$. Then since the $\alpha-$H\"older seminorm of $\phi'$ is bounded by $M$ we have
$$
\max_{y\neq 0}\frac{|\phi'(0)-\phi'(y)|}{|y|^\alpha}<M\implies
-M|y|^\alpha\leq \phi'(0)-\phi'(y)\leq M|y|^\alpha\implies
\phi'(y)\geq \phi'(0)-M|y|^\alpha.
$$
Putting this together with
$$\frac{\phi(x)}{x}=\frac{1}{x}\int_0^x\phi'(y)dy,$$
we find
$$
\frac{\phi(x)}{x}\geq\frac{1}{x}\int_0^x(\varepsilon'-My^\alpha)dy
=\varepsilon'-\frac{Mx^\alpha}{\alpha+1}.$$
It follows that
$$
\varepsilon'\leq \frac{M}{\alpha+1}|x|^\alpha+\frac{|\phi(x)|}{|x|}\leq \frac{M}{\alpha+1}|x|^\alpha+\varepsilon\rightarrow \frac{M}{\alpha+1}\delta^\alpha+\varepsilon,$$
where in the last line we let $|x|\rightarrow\delta$.
Finally if $|\phi|_{C^0}$ is attained at some other point $x_0\in S^1$, then simply apply the preceding argument to 
$\Tilde{\phi}(x)=\phi(x+x_0)-\phi(0)$.
\end{proof}
To finish the proof of Theorem 2.2, we will apply Lemma 2.2 to the function $F={f}_N-f$. Let $\varepsilon=2K'\lambda^{N/2}$ and $\delta=\lambda^{N/2}$. Then
$$
\sup_{|x-y|>\delta}\frac{|F(x)-F(y)|}{|x-y|}<\frac{2|F|_{C^0}}{\delta}\leq \frac{2K'd_{C^0}(h,{h}_N)}{\lambda^{N/2}}\leq
\frac{2K'\lambda^N}{\lambda^{N/2}}=2K'\lambda^{N/2}.
$$
It remains to prove that $|F'|_{C^\alpha}$ is uniformly bounded for $f\in\mathcal{W}$ for every $\alpha<1$. We have
$$
|F'|_{C^\alpha}\leq |f'|_{C^\alpha}+|{f}_N'|_{C^\alpha},
$$
and since $f$ is uniformly bounded in the $C^2$ seminorm, it will be uniformly bounded in the $C^{1+\alpha}$ seminorm. So it remains to uniformly bound 
$|{f}_N'|_{C^\alpha}=|(h_N^{-1}\circ g\circ h_N)'|_{C^\alpha}=|((h_N^{-1})'\circ g\circ h_N)(g'\circ h_N)h_N'|_{C^\alpha}$. By the product rule for the $\alpha-$H\"older seminorm and symmetry between $h_N=I_g^{-1}\circ I_f$ and $h_N^{-1}=I_f^{-1}\circ I_g$, it suffices to uniformly bound $|h_N'|_{C^\alpha}=|\frac{\rho_f}{\rho_g\circ h_N}|_{C^\alpha}$. By Lemma 2.1, we can uniformly bound $|\rho_f|_{C^\alpha}$ for $f\in\mathcal{W}$ and hence by properties of the  $\alpha-$H\"older seminorm, we can uniformly bound $|\frac{\rho_f}{\rho_g\circ h_N}|_{C^\alpha}.$ Therefore, we may apply Lemma 2.2 and we find that 
$|F|_{C^1}\leq\frac{M}{\alpha+1}\delta^\alpha+\varepsilon=
\frac{M}{\alpha+1}(2K'\lambda^{\alpha N/2})+\lambda^{N/2}=K''\lambda^{\alpha N/2}$. For $\alpha<1$ such that $\lambda^{\alpha/2}=\lambda_0$, we get the desired conclusion.
\end{proof}
\begin{cor}
    Let $f,g\in\expand^{r+1+\alpha}_\gamma(S^1)$ for $r\in\mathbb{N}$, $r\geq 2$, and suppose that $d_{C^{r+1+\alpha}}(f,g)<C_0$. Under the same hypotheses of Theorem 1.3, we have that there exists a constant $K_r>0$ independent of $f$ such that  $d_{C^r}(f,{f}_N)\leq K_r\lambda^{2^{-r}N}$.
\end{cor}
\begin{proof}
     We proceed by induction. The base case $r=2$ follows exactly as in Theorem 1.3, except the added assumption that our systems are $C^{2+\alpha}$ allow us to get uniform bounds on $|F'|_{C^1}$ (using the argument of Lemma 2.1) and apply Lemma 2.2 with $\alpha=1$. Letting $F={f}_N-f$, we assume by induction that we have proven $|F|_{C^r}\leq K_r\lambda^{2^{-r}N}$. To obtain a similar estimate for $|F|_{C^{r+1}}$, we apply Lemma 2.2 to the function $F^{(r)}$, with $\varepsilon=2K_r\lambda^{2^{-r-1}N}$ and $\delta=\lambda^{2^{-r-1}N}.$ For these choices we find 
    $$
\sup_{|x-y|>\delta}\frac{|F^{(r)}(x)-F^{(r)}(y)|}{|x-y|}<\frac{2|F|_{C^r}}{\lambda^{2^{-r}N}}\leq \frac{2K_{r}\lambda^{2^{-r}N}}{\lambda^{2^{-r-1}N}}=K_r\lambda^{2^{-r-1}N}.$$
The assumption that are maps are $C^{r+2+\alpha}$ is so that we can compactly embed the set $\mathcal{W}$ in $C^{r+2}$, thereby getting uniform bounds on
$|F^{(r+1)}|_{C^1}$.
The conclusion now follows Lemma 2.2 exactly as in the proof of Theorem 1.3.
\end{proof}
The next corollary establishes a similar estimate on the exponential decay of $d_{C^r}(f,{f}_N)$ without such a loss of exponent under the stronger assumption that $f$ and $g$ are close in the $C^k$-topology for all $k$.
\begin{cor}
    Let $f,g\in\expand^{\infty}_\gamma(S^1)$ and suppose that $\sup_{k\geq 0}d_{C^k}(f,g)<C_0$. Then for any $r\in\mathbb{N}$ and any $0<\lambda^{1/2}<\lambda_0<1$, there exists a constant $K_r'>0$ such that under the hypotheses of Theorem 1.3, we have $d_{C^r}(f,{f}_N)\leq K_r'\lambda_0^N$.
\end{cor}
\begin{proof}
    The proof follows from interpolation theory on the spaces $C^r(S^1)$ (see \cite{Interpolation} Remark 1.22). For $k_1<m<k_2\in\mathbb{N}$ we have that 
    $||\phi||_{C^m}\leq C_{k_1,k_2,m}||\phi||_{C^{k_1}}^{1-t}||\phi||_{C^{k_2}}^t$ for any $\phi\in C^{k_2}(S^1)$, where $t=\frac{m-k_1}{k_2-k_1}$. We will apply this with $\phi=F={f}_N-f,$ $k_1=1$, and by taking $k_2$ sufficiently large, we have that $t$ can be made arbitrarily close to $0$. We choose $k_2$ so that $\lambda^{\frac{1-t}{2}}\leq \lambda_0$. We want to bound the term $||F||_{C^{k_2}}$ using the bound $d_{C^{k_2+2}}(f,g)<C_0$. Applying Corollary 2.1 gives us 
    $||F||_{C^{k_2}}\leq K_{k_2}\lambda^{\frac{tN}{2^{k_2}}}\leq K_{k_2}$. Hence by our interpolation inequality (collecting all constants into $C$):
    $$
    ||F||_{C^m}\leq C_{k_1,k_2,m}||F||_{C^{k_1}}^{1-t}||F||_{C^{k_2}}^t\leq C_{k_1,k_2,m} K_{1}^{1-t}K_{k_2}^t\left(\lambda^{\frac{1-t}{2}}\right)^N\leq C\lambda_0^N.
    $$
\end{proof}

\section{Effective Equidistribution}
We begin by reviewing the basic definitions of subshifts of finite type and transfer operators. For a more detailed treatment, see \cite{Baladi} and \cite{Parry_Pollicott}.
Let $A$ be an irreducible and aperiodic $0,1$-matrix, which we will refer to as a transition matrix, and consider the set
$$
\Sigma_A^+:=\{x\in\{1,\cdots,s\}^{\mathbb{Z}_{\geq0}}\hspace{1mm}|\hspace{1mm} A(x_i,x_{i+1})=1,\forall i\geq 0\}.
$$
We interpret $\Sigma_A^+$ as the set of all sequences in $s$-symbols that are allowed by the transition matrix $A$. Consider the left shift map $\sigma_A^+:\Sigma_A^+\rightarrow\Sigma_A^+$ defined by 
$(\sigma_A^+(x))_n=x_{n+1}$. We refer to the dynamical system 
$(\Sigma_A^+,\sigma_A^+)$ as a one-sided subshift of finite type. Analogously, we define the two-sided subshift of finite type to be the system $(\Sigma_A,\sigma_A)$, where
$$
\Sigma_A:=\{x\in\{1,\cdots,s\}^{\mathbb{Z}}\hspace{1mm}|\hspace{1mm} A(x_i,x_{i+1})=1,\forall i\in\mathbb{Z}\},
$$
and $(\sigma_A(x))_n=x_{n+1}$.
The key difference between the two systems is that for the one-sided subshift of finite type, $\sigma_A^+$ is non-invertible, whereas for the two-sided subshift of finite type, $\sigma_A$ is invertible.
When our transition matrix is clear, and when it is also clear whether we are talking about a one-sided or two-sided shift, we shall denote the left shift map simply as $\sigma$. 

We topologize $\Sigma_A^+$ with the metric 
$d_{\theta}(x,y)=\theta^{\max\{n\geq0|x_i=y_i, 0\leq i<n\}}$, where $0<\theta<1$ is a fixed constant. Notice that with respect to this metric, $\sigma$ is a $\theta$-expansion. Let $\mathcal{F}_\theta^+$ denote the Banach space of all functions $\phi:\Sigma_A^+\rightarrow\mathbb{C}$ which are Lipschitz continuous with respect to this metric. Denote by $|\cdot|_\theta$ and $||\cdot||_\theta$ the Lipschitz seminorm and Lipschitz norm with respect to this metric, respectively. We analogously, topologize $\Sigma_A$ and define $\mathcal{F}_\theta$.
\begin{definition}
    Fix a weight function $\psi\in\mathcal{F}_\theta^+$ and define the Ruelle transfer operator $\mathcal{L}_\psi:\mathcal{F}_\theta^+\rightarrow\mathcal{F}_\theta^+$ by the formula

$$
\mathcal{L}_\psi(\phi)(x)=\sum_{\sigma(y)=x}e^{\psi(y)}\phi(y).
$$
\end{definition}

By the Ruelle-Perron-Frobenius Theorem (see \cite{Baladi} theorem 1.5), the operator $\mathcal{L}_\psi$ is quasi-compact, with a unique maximal positive simple eigenvalue $\lambda=e^{P(\psi)}$ corresponding to a strictly positive eigenfunction $\rho$, and all other points of the spectrum lie in a strictly smaller disc.
Let us further assume that the transfer operator is normalized so that $\mathcal{L}_\psi(1)=e^{P(\psi)}$ (which can always be accomplished by replacing the weight $\psi$ with $\overline{\psi}=\psi+\log(\rho)-\log(\rho)\circ\sigma$, and observing that $P(\psi)=P(\overline{\psi}))$. Then the eigenmeasure $\mu_\psi$ corresponding to the eigenvalue $e^{P(\psi)}$ of the dual operator $\mathcal{L}_\psi^*$ is the unique equilibrium state of the potential $\psi$ . Observe that 

$$
\mathcal{L}_\psi^n(\phi)(x)=\sum_{\sigma^n(y)=x}e^{S_n\psi(y)}\phi(y).
$$
What follows is the analog of Theorem 2.2 for subshifts of finite type.

\begin{theorem}[Effective Equidistribution for Equilibrium States]
Let $(\Sigma_A^+,\sigma_A^+)$ be a subshift of finite type, where the transition matrix $A$ is irreducible and aperiodic, and let $\psi\in\mathcal{F}_\theta^+$ be a Lipschitz continuous potential. Then there exists constants $C>0$ and $0<\tau<1$ such that for any $\phi\in\mathcal{F}_\theta^+$ and all $n\in\mathbb{N}$, 

$$
\bigg|\int \phi d\mu_{\psi,n}-\int \phi d\mu_\psi\bigg|\leq
C||\phi||_\theta\tau^n,
$$
where $\mu_\phi$ is the unique equilibrium state of $\psi$.
\end{theorem}

\begin{proof}
By replacing $\phi$ by $\phi-\int\phi d\mu_\psi$, we may assume that $\int\phi d\mu_\psi=0$. We need to show that 
$$
\bigg|\frac{1}{Z_n}\sum_{\sigma^n(x)=x}e^{S_n\psi(x)}\phi(x)\bigg|\leq C||\phi||_\theta\tau^n,
$$
where $Z_n$ is the normalization constant. By \cite{katok_hasselblatt_1995} Proposition 20.3.3, there exists a constant $D>0$ such that
$\frac{1}{D}e^{nP(\psi)}\leq Z_n\leq De^{nP(\psi)}$ (an inspection of the proof reveals that this constant $D$ can be made uniform in Theorem 2.2).
Let $[i]=\{x\in\Sigma_A^+|x_0=i\}$, and for a string
$\underline{i}=(i_0,\cdots,i_{n-1})$ let us denote its length by $|\underline{i}|=n$ and its cylinder set by
$[\underline{i}]=\{x\in\Sigma_A^+|x_0=i_0,\cdots,x_{n-1}\}$. For each $1\leq i\leq s$, fix any point $x_i\in[i]$, and for each string $\underline{i}$, fix a point of period $n$, $x_{\underline{i}}\in [\underline{i}]$, if one exists, and let $x_{\underline{i}}\in [\underline{i}]$ be arbitrary otherwise. We first claim that
$$
\sum_{\sigma^n(x)=x}e^{S_n\psi(x)}\phi(x)=\sum_{|\underline{i}|=n}\mathcal{L}_\psi^n(\chi_{[\underline{i}]}\phi)(x_{\underline{i}}).
$$
To see this, we expand out each term in the right sum:
$$
\mathcal{L}_\psi^n(\chi_{[\underline{i}]}\phi)(x_{\underline{i}})=\sum_{\sigma^n(y)=x_{\underline{i}}}e^{S_n\psi(y)}\chi_{[\underline{i}]}(y)\phi(y)
=e^{S_n\psi(\underline{i}x_{\underline{i}})}\phi(\underline{i}x_{\underline{i}})=e^{S_n\psi(x_{\underline{i}})}\phi(x_{\underline{i}}),
$$
where $\underline{i}x=(i_0,\cdots,i_{n-1},x_0,x_1,\cdots)$ denotes the only inverse branch of $\sigma^n$ that contributes to the sum due to the characteristic function. Since each point of $\fix(\sigma^n)$ lies in a unique cylinder set $[\underline{i}]$, and each such cylinder set contains at most one period-$n$ point, we see that all periodic points are accounted for in the sum over $|\underline{i}|=n$. Consider the estimate
$$
\bigg|\sum_{\sigma^n(x)=x}e^{S_n\psi(x)}\phi(x)\bigg|\leq\bigg| \sum_{\sigma^n(x)=x}e^{S_n\psi(x)}\phi(x)-\sum_{i=1}^s\mathcal{L}_\psi^n(\chi_{[i]}\phi)(x_i)\bigg|+\bigg|\sum_{i=1}^s\mathcal{L}_\psi^n(\chi_{[i]}\phi)(x_i)\bigg|.
$$
To estimate both terms on the right, we will first decompose the transfer operator as a sum of its projection onto the eigenspace of $e^{P(\psi)}$, and the orthogonal projection:
$\mathcal{L}_\psi=\mathcal{P}+\mathcal{N}$, where
$\mathcal{P}(\phi)=e^{P(\psi)}\int\phi d\mu_\psi$, and $\mathcal{N}$ has spectral radius $re^{P(\psi)}$ with $0<r<1$. The second term can thus be estimated as
$$
\bigg|\sum_{i=1}^s\mathcal{L}_\psi^n(\chi_{[i]}\phi)(x_i)\bigg|\leq 
\bigg|\sum_{i=1}^s\bigg[\mathcal{P}^n(\chi_{[i]}\phi)(x_i)+\mathcal{N}^n(\chi_{[i]}\phi)(x_i)\bigg]\bigg|$$
$$
=\bigg|\sum_{i=1}^s\bigg[e^{nP(\psi)}\int_{[i]}\phi d\mu_\psi+\mathcal{N}^n(\chi_{[i]}\phi)(x_i)\bigg]\bigg|
$$
$$
=\bigg|e^{nP(\psi)}\int\phi d\mu_\psi+\sum_{i=1}^s\mathcal{N}^n(\chi_{[i]}\phi)(x_i)\bigg|=\bigg|\sum_{i=1}^s\mathcal{N}^n(\chi_{[i]}\phi)(x_i)\bigg|\leq$$
$$
\sum_{i=1}^s||\mathcal{N}^n||||\chi_{[i]}\phi||_\theta\leq
C(r+\epsilon)^ne^{nP(\psi)}||\phi||_\theta,
$$
where $\epsilon>0$ is arbitrary and $C>0$ depends on $\epsilon$.

It remains to estimate
$$
\bigg| \sum_{\sigma^n(x)=x}e^{S_n\psi(x)}\phi(x)-\sum_{i=1}^s\mathcal{L}_\psi^n(\chi_{[i]}\phi)(x_i)\bigg|
=\bigg| \sum_{|\underline{i}|=n}\mathcal{L}_\psi^n(\chi_{[\underline{i}]}\phi)(x_{\underline{i}})-\sum_{i=1}^s\mathcal{L}_\psi^n(\chi_{[i]}\phi)(x_i)\bigg|
$$
If $\underline{i}=(i_0,\cdots,i_{n-1})$, we let $\underline{j}(\underline{i})=(i_0,\cdots,i_{n-2})$. We now telescope our above series:
$$
 \sum_{|\underline{i}|=n}\mathcal{L}_\psi^n(\chi_{[\underline{i}]}\phi)(x_{\underline{i}})-\sum_{i=1}^s\mathcal{L}_\psi^n(\chi_{[i]}\phi)(x_i)$$
 
 $$=\sum_{m=2}^n\left(\sum_{|\underline{i}|=m}\mathcal{L}_\psi^n(\chi_{[\underline{i}]}\phi)(x_{\underline{i}})
 -\sum_{|\underline{i}|=m-1}\mathcal{L}_\psi^n(\chi_{[\underline{j}]}\phi)(x_{\underline{j}})\right)
$$ 
$$ 
=\sum_{m=2}^n\sum_{|\underline{i}|=m}\left(\mathcal{L}_\psi^n(\chi_{[\underline{i}]}\phi)(x_{\underline{i}})-\mathcal{L}_\psi^n(\chi_{[\underline{i}]}\phi)(x_{\underline{j}(\underline{i})})\right)$$
$$
=\sum_{m=2}^n\sum_{|\underline{i}|=m}\left((\mathcal{P}^n+\mathcal{N}^n)(\chi_{[\underline{i}]}\phi)(x_{\underline{i}})-(\mathcal{P}^n+\mathcal{N}^n)(\chi_{[\underline{i}]}\phi)(x_{\underline{j}(\underline{i})})\right)$$
$$
=\sum_{m=2}^n\sum_{|\underline{i}|=m}\left(\mathcal{N}^n(\chi_{[\underline{i}]}\phi)(x_{\underline{i}})-\mathcal{N}^n(\chi_{[\underline{i}]}\phi)(x_{\underline{j}(\underline{i})})\right).
$$
We now take the absolute value of both sides, and observe that $d_\theta(x_{\underline{i}},x_{\underline{j}(\underline{i})}) =\theta^{m-1}$:
$$
\bigg|\sum_{m=2}^n\sum_{|\underline{i}|=m}\left(\mathcal{N}^n(\chi_{[\underline{i}]}\phi)(x_{\underline{i}})-\mathcal{N}^n(\chi_{[\underline{i}]}\phi)(x_{\underline{j}(\underline{i})})\right)\bigg|\leq$$
$$
\sum_{m=2}^n\sum_{|\underline{i}|=m}||\mathcal{N}^n\chi_{[\underline{i}]}\phi||_\theta\theta^{m-1}\leq\sum_{m=2}^n||\mathcal{N}^{n-m}||\sum_{|\underline{i}|=m}||\mathcal{L}_\psi^m\chi_{[\underline{i}]}\phi||_\theta\theta^{m-1}.
$$
In this last inequality we made use of the fact that $\mathcal{N}\mathcal{P}=0$ to get the bound
$$||\mathcal{N}^n(\chi_{[\underline{i}]}\phi)||_\theta=||\mathcal{N}^{n-m}(\mathcal{N}^{m}(\chi_{[\underline{i}]}\phi)+\mathcal{P}^{m}(\chi_{[\underline{i}]}\phi))||_\theta=
$$
$$
||\mathcal{N}^{n-m}(\mathcal{L}_\psi^{m}(\chi_{[\underline{i}]}\phi))||_\theta\leq ||\mathcal{N}^{n-m}||||\mathcal{L}_\psi^{m}(\chi_{[\underline{i}]}\phi))||_\theta.$$ 
We next estimate the term
$||\mathcal{L}_\psi^{m}(\chi_{[\underline{i}]}\phi))||_\theta=||e^{(S_m\psi)\circ\sigma_{\underline{i}}^{-1}}(\phi\circ\sigma_{\underline{i}}^{-1})||_\theta$, where $\sigma_{\underline{i}}^{-1}(x)=\underline{i}x$ is an inverse branch of $\sigma^n$:
$$
||e^{(S_m\psi)\circ\sigma_{\underline{i}}^{-1}}(\phi\circ\sigma_{\underline{i}}^{-1})||_\theta=
|e^{(S_m\psi)\circ\sigma_{\underline{i}}^{-1}}(\phi\circ\sigma_{\underline{i}}^{-1})|_{\infty}+|e^{(S_m\psi)\circ\sigma_{\underline{i}}^{-1}}(\phi\circ\sigma_{\underline{i}}^{-1})|_\theta\leq
$$
$$
|e^{(S_m\psi)\circ\sigma_{\underline{i}}^{-1}}|_{\infty}|\phi|_{\infty}+|e^{(S_m\psi)\circ\sigma_{\underline{i}}^{-1}}|_{\infty}|(\phi\circ\sigma_{\underline{i}}^{-1})|_\theta+
|e^{(S_m\psi)\circ\sigma_{\underline{i}}^{-1}}|_\theta|(\phi\circ\sigma_{\underline{i}}^{-1})|_{\infty}\leq
$$
$$
|e^{(S_m\psi)\circ\sigma_{\underline{i}}^{-1}}|_{\infty}|\phi|_{\infty}+|e^{(S_m\psi)\circ\sigma_{\underline{i}}^{-1}}|_{\infty}|(\phi\circ\sigma_{\underline{i}}^{-1})|_\theta+
|e^{(S_m\psi)\circ\sigma_{\underline{i}}^{-1}}|_{\infty}|S_m\psi\circ\sigma_{\underline{i}}^{-1}|_\theta|(\phi\circ\sigma_{\underline{i}}^{-1})|_{\infty}=
$$
$$
|e^{(S_m\psi)\circ\sigma_{\underline{i}}^{-1}}|_{\infty} \left(
|\phi|_{\infty}+|\phi\circ\sigma_{\underline{i}}^{-1}|_\theta
+|S_m\psi\circ\sigma_{\underline{i}}^{-1}|_\theta|(\phi\circ\sigma_{\underline{i}}^{-1})|_{\infty}\right)\leq
$$
$$
|e^{(S_m\psi)\circ\sigma_{\underline{i}}^{-1}}|_{\infty} \left(
||\phi||_{\theta}+|\phi\circ\sigma_{\underline{i}}^{-1}|_\theta
+|S_m\psi\circ\sigma_{\underline{i}}^{-1}|_\theta||\phi\||_{\theta}\right).
$$
To calculate the seminorms $|\phi\circ\sigma_{\underline{i}}^{-1}|_\theta$ and
$|S_m\psi\circ\sigma_{\underline{i}}^{-1}|_\theta$, we use the fact that $\sigma_{\underline{i}}^{-1}$ is a $\theta^m$-contraction:
$$
|S_m\psi\circ\sigma_{\underline{i}}^{-1}|_\theta=
\sup_{x\neq y}\frac{|S_m\psi\circ\sigma_{\underline{i}}^{-1}(x)-S_m\psi\circ\sigma_{\underline{i}}^{-1}(y)|}{d_\theta(x,y)}\leq$$
$$
\sum_{i=0}^{m-1}\frac{|\psi(\sigma^i(\sigma_{\underline{i}}^{-1}(x)))-\psi(\sigma^i(\sigma_{\underline{i}}^{-1}(y)))|}{d_\theta(x,y)}\leq
\sum_{i=0}^{m-1}|\psi|_\theta \theta^{m-i}\leq\frac{1}{1-\theta}|\psi|_\theta.
$$
A similar (and simpler) calculation shows that 
$|\phi\circ\sigma_{\underline{i}}^{-1}|_\theta\leq\theta^m||\phi||_\theta$.
Putting this all together, we find that
$$||\mathcal{L}_\psi^m\chi_{[\underline{i}]}\phi||_\theta\theta^{m-1}\leq
|e^{(S_m\psi)\circ\sigma_{\underline{i}}^{-1}}|_{\infty}\left(
||\phi||_\theta\theta^{m-1}+||\phi||_\theta\theta^{2m-1}+
\frac{|\psi|_\theta}{1-\theta}||\phi||_\theta\theta^m\right)$$
$$
\leq C'|e^{(S_m\psi)\circ\sigma_{\underline{i}}^{-1}}|_{\infty}||\phi||_\theta\theta^m,
$$
where $C'=\max\{1,\frac{|\psi|_\theta}{1-\theta}\}$. Thus, using the spectral radius bound for $\mathcal{N}$, we find that
$$
\sum_{m=2}^n||\mathcal{N}^{n-m}||\sum_{|\underline{i}|=m}||\mathcal{L}_\psi^m\chi_{[\underline{i}]}\phi||_\theta\theta^{m-1}$$
$$
\leq
\sum_{m=2}^nC(r+\epsilon)^{n-m}e^{(n-m)P(\psi)}C'\theta^m||\phi||_\theta\sum_{|\underline{i}|=m}|e^{(S_m\psi)\circ\sigma_{\underline{i}}^{-1}}|_{\infty}$$
$$
\leq C''||\phi||_\theta\kappa^n\sum_{m=2}^ne^{(n-m)P(\psi)}\sum_{|\underline{i}|=m}|e^{(S_m\psi)\circ\sigma_{\underline{i}}^{-1}}|_{\infty},
$$
where $\kappa=\max\{\theta,r+\epsilon\}<1$.
We have seen previously that $|S_m\psi\circ\sigma_{\underline{i}}^{-1}(x)-S_m\psi\circ\sigma_{\underline{i}}^{-1}(y)|\leq\frac{|\psi|_\theta}{1-\theta}d_\theta(x,y)\leq K<\infty$ (since the diameter of $\Sigma_A^+$ is finite). Taking the exponential of both sides, we obtain the following bounded distortion estimate:
$$
e^{(S_m\psi)\circ\sigma_{\underline{i}}^{-1}(x)}\leq Ce^{(S_m\psi)\circ\sigma_{\underline{i}}^{-1}(y)},
$$
for any $x,y$ in the domain of $\sigma_{\underline{i}}^{-1}$. Notice that the domain of $\sigma_{\underline{i}}^{-1}$ is completely determined by the last symbol in the string $\underline{i}$. For each $\underline{i}$, let $y_{\underline{i}}$ be such that $e^{(S_m\psi)\circ\sigma_{\underline{i}}^{-1}(y_{\underline{i}})}=|e^{(S_m\psi)\circ\sigma_{\underline{i}}^{-1}}|_{\infty}$, and let $z_{\underline{i}}$ be any point in the domain of  $\sigma_{\underline{i}}^{-1}$ that only depends on the last symbol $i_{m}$.
Then
$$
\sum_{|\underline{i}|=m}|e^{(S_m\psi)\circ\sigma_{\underline{i}}^{-1}}|_{\infty}\leq C\sum_{|\underline{i}|=m}e^{(S_m\psi)\circ\sigma_{\underline{i}}^{-1}(z_{\underline{i}})}=C\sum_{i_m=1}^s\mathcal{L}_\psi^m1(z_{\underline{i}})\leq Cs||\mathcal{L}_\psi^m 1||_\theta\leq Ce^{m(P(\psi)+\epsilon)}.
$$
Therefore,
$$
=\bigg| \sum_{|\underline{i}|=n}\mathcal{L}_\psi^n(\chi_{[\underline{i}]}\phi)(x_{\underline{i}})-\sum_{i=1}^s\mathcal{L}_\psi^n(\chi_{[i]}\phi)(x_i)\bigg|\leq
C''||\phi||_\theta\kappa^n\sum_{m=2}^ne^{(n-m)P(\psi)}\sum_{|\underline{i}|=m}|e^{(S_m\psi)\circ\sigma_{\underline{i}}^{-1}}|_{\infty}\leq$$
$$
C'''||\phi||_\theta\kappa^ne^{n(P(\psi)+\epsilon)}(n-2)\leq
C'''||\phi||_\theta(\kappa+\epsilon)^ne^{n(P(\psi)+\epsilon)}.
$$
We have shown that
$$
\bigg|\sum_{\sigma^n(x)=x}e^{S_n\psi(x)}\phi(x)\bigg|\leq
C||\phi||_\theta(\kappa+\epsilon)^ne^{n(P(\psi)+\epsilon)}
$$
with $\kappa+\epsilon<1$. Finally,
$$
\bigg|\frac{1}{Z_n}\sum_{\sigma^n(x)=x}e^{S_n\psi(x)}\phi(x)\bigg|\leq De^{-nP(\psi)}C||\phi||_\theta(\kappa+\epsilon)^ne^{n(P(\psi)+\epsilon)}\leq C||\phi||_\theta(\kappa+\epsilon)^ne^{n\epsilon}.
$$
Finally, we let $\tau=(\kappa+\epsilon)e^\epsilon<1$, for sufficiently small $\epsilon>0$.
\end{proof}
Using standard arguments, it is easy to deduce the same equidisribution result for two-sided shifts:
\begin{cor}
Let $(\Sigma_A,\sigma_A)$ be a subshift of finite type, where the transition matrix $A$ is irreducible and aperiodic, and let $\psi\in\mathcal{F}_\theta$ be a Lipschitz continuous potential. Then there exists constants $C>0$ and $0<\tau<1$ such that for any $\phi\in\mathcal{F}_\theta$ and all $n\in\mathbb{N}$, 

$$
\bigg|\int \phi d\mu_{\psi,n}-\int \phi d\mu_\psi\bigg|\leq
C||\phi||_\theta\tau^n,
$$
where $\mu_\phi$ is the unique equilibrium state of $\psi$.
\end{cor}
\begin{proof}
By \cite{Baladi} Lemma 1.3, there exists functions $\psi^+,\phi^+\in\mathcal{F}_{\theta^{1/2}}$ which depend only on future coordinates and are cohomologous to $\psi$ and $\phi$, respectively. Therefore we may regard $\psi^+$ and $\phi^+$ as functions in $\mathcal{F}_{\theta^{1/2}}^+$.
Since they are cohomologous, $\psi$ and $\psi^+$ have the same equilibrium state, which may be regarded as a measure on $\Sigma_A^+$, and the dynamical sums agree on periodic points ($S_n\psi(x)=S_n\psi^+(x)$) so that $\mu_{\psi,n}=\mu_{\psi^+,n}$. Therefore,
$$
\bigg|\int \phi d\mu_{\psi,n}-\int \phi d\mu_\psi\bigg|=
\bigg|\int \phi^+ d\mu_{\psi^+,n}-\int \phi^+ d\mu_\psi^+\bigg|\leq
C||\phi^+||_{\theta^{1/2}}\tau^n.
$$
Moreover, the map $\phi\mapsto\phi^+$ is a continuous, linear mapping from $\mathcal{F}_\theta\rightarrow\mathcal{F}_{\theta^{1/2}}^+$, so $||\phi^+||_{\theta^{1/2}}\leq C'||\phi||_\theta$. Absorbing the constant $C'$ into $C$ gives the desired result.
\end{proof}
To pass to the proof of Theorem 2.2 we will use Markov partitions. 
It is well known that repellers and Axiom A diffeomorphisms admit finite Markov partitions; see \cite{10.5555/1875355} Theorem 3.5.2 in the case of expanding maps and \cite{Bowen-1970} for the case of Axiom A diffeomorhisms. 
More precisely, if $(J,T)$ is a repeller, or $(\Omega(T),T)$ is an Axiom A diffeomorphism, then there exists a subshift of finite type $(\Sigma_A^+,\sigma)$ and a semiconjugacy $\pi:\Sigma_A^+\rightarrow J$ (respectively $(\Sigma_A,\sigma)$ for $(\Omega(T),T)$). If the transformation $T$ is mixing, then the transition matrix $A$ is irreducible and aperiodic. For an appropriately chosen $\theta$, Lipschitz functions $f$ defined on $J$ or $\Omega(T)$ can be lifted to a Lipschitz function $f\circ\pi$. We proceed to prove Theorem 2.2 in the case of one dimensional expanding maps, where the passage to a Markov partition is simplest. However, in order to obtain a uniform effective equidistribution for expanding maps, we need uniform effective equidistribution of corresponding lifted systems in Theorem 3.1. Uniformity is lost in Theorem 3.1 when we apply the spectral radius formula to get bounds on $||\mathcal{N}^n||_\theta$. However, for the particular systems we are considering, we can obtain uniform bounds on $||\mathcal{N}^n||_\theta$ by using the Birkhoff cone technique, adapted to subshifts of finite type by Naud in \cite{Frédéric}.
\begin{lemma}
 There exists $C_{\mathcal{W}}>0$ and $0<\tau_{\mathcal{W}}<1$ such that the following is true:
 Given $f\in\mathcal{W}$, let $\pi_f$ be the associated semiconjugacy to the full shift on $\deg(f)$-symbols, and let $\psi_f=-\log(f'\circ\pi_f)$. If $\mathcal{L}_{\psi_f}h_f= e^{P(\psi_f)}h_f$, we consider the normalized potential $\overline{\psi}_f=\psi_f+\log(h_f)-\log(h_f\circ\sigma)$. Then for all $n\in\mathbb{N}$, $||\mathcal{N}_{\overline{\psi}_f}^n||_\theta< C_{\mathcal{W}}\tau_{\mathcal{W}}^ne^{nP(\overline{\psi}_f)}$.
\end{lemma}
We review the Birkhoff cone construction and proof Lemma 3.1 in the appendix.

\begin{proof}[Proof of Theorem 2.2]
Consider a degree $k$ expanding map $f:S^1\rightarrow S^1$. Then the Markov partition of $(S^1,f)$ consists of $k$ closed intervals which are each mapped under $f$ to the entirety of $S^1$. Consequentially, $(S^1,f)$ is semiconjugated to the full shift on $k$ symbols, $(\Sigma^+,\sigma)$. Let $\phi:S^1\rightarrow\mathbb{C}$ be a Lipschitz function. We can then lift $\phi$ to a Lipschitz function $f\circ\pi:\Sigma^+\rightarrow\mathbb{C}$. Likewise, let $\psi:S^1\rightarrow\mathbb{C}$ be a H\"older continuous potential with equilibrium state $\mu_\psi$, and let
$\psi\circ\pi:\Sigma^+\rightarrow\mathbb{C}$ be the lifted potential with equilibrium state $\mu_{\psi\circ\pi}$.
We claim that the pushforward of $\mu_{\psi\circ\pi}$ under $\pi$ is $\mu_\psi$, i.e.,  $\pi_*\mu_{\psi\circ\pi}=\mu_\psi$.
Unfortunately, it is not true that $\pi_*\mu_{\psi\circ\pi}^n=\mu_\psi^n$. since the semiconjugacy is not injective, we may have multiple distinct periodic oribits for $(\Sigma^+,\sigma)$ get mapped under $\pi$ to the same periodic orbit of $(S^1,f)$.

We know that $|\fix(\sigma^n)|=k^n$ and $|\fix(f^n)|=k^n-1$. Moreover, distinct periodic orbits of $f$ can be lifted to distinct periodic orbits of $\sigma$ with the same period, so we have that for each $n$, only two distinct points of $\Sigma^+$ of period $n$ get mapped under $\pi$ to the same point. Let $\nu_{\psi\circ\pi}^n$ be any measure on $\Sigma^+$ such that $\pi^*\nu_{\psi\circ\pi}^n=\mu_{\psi}^n$. Then

$$
\bigg|\int\phi d\mu_\psi^n-\int\phi d\mu_\psi\bigg|=
\bigg|\int\phi\circ\pi d\nu_{\psi\circ\pi}^n-\int\phi\circ\pi d\mu_{\psi\circ\pi}\bigg|$$
$$
\leq\bigg|\int\phi\circ\pi d\mu_{\psi\circ\pi}^n-\int\phi\circ\pi d\mu_{\psi\circ\pi}\bigg|+
\bigg|\int\phi\circ\pi d\nu_{\psi\circ\pi}^n-\int\phi\circ\pi d\mu_{\psi\circ\pi}^n\bigg|$$
To estimate the last term, we write
$$
\nu_{\psi\circ\pi}^n=\frac{1}{Z_n(\psi)}\sum_{x\in A_n} e^{S_n(\psi)(\pi(x))}\delta_x,$$
where $A_n$ is any subset of $\fix(\sigma^n)$ that is mapped bijectively to $\fix(f^n)$, and $Z_n({\psi})$ is a normalization constant. Then if we let $y\in\fix(\sigma^n)/A_n$ we have
$$
\bigg|\int\phi\circ\pi d\nu_{\psi\circ\pi}^n-\int\phi\circ\pi d\mu_{\psi\circ\pi}^n\bigg|=
$$
$$
\bigg|\frac{1}{Z_n(\psi)}\sum_{x\in A_n} e^{S_n(\psi)(\pi(x))}\phi(\pi(x))-\frac{1}{Z_n(\psi\circ\pi)}\sum_{x\in \fix(\sigma^n)} e^{S_n(\psi)(\pi(x))}\phi(\pi(x))\bigg|\leq
$$
$$
\bigg|\left(\frac{1}{Z_n(\psi)}-\frac{1}{Z_n(\psi\circ\pi)}\right)\sum_{x\in \fix(\sigma^n)} e^{S_n(\psi)(\pi(x))}\phi(\pi(x))\bigg|$$

$$+
\frac{1}{Z_n(\psi)}\bigg|\sum_{x\in A_n} e^{S_n(\psi)(\pi(x))}\phi(\pi(x))-\sum_{x\in \fix(\sigma^n)} e^{S_n(\psi)(\pi(x))}\phi(\pi(x))\bigg|
$$
$$
\leq \left(\frac{Z_n(\psi\circ\pi)}{Z_n(\psi)}-1\right)||\phi\circ\pi||_{\infty}+\frac{1}{Z_n(\psi)}e^{S_n(\phi)(\pi(y))}|\phi(\pi(y))|\leq
$$
$$
\left(\frac{Z_n(\psi)+e^{S_n(\phi)(\pi(y))}}{Z_n(\psi)}-1\right)||\phi\circ\pi||_{\infty}+\frac{1}{Z_n(\psi)}e^{S_n(\phi)(\pi(y))}|\phi(\pi(y))|\leq
$$
$$
\frac{2}{Z_n(\psi)}e^{S_n(\phi)(\pi(y))}||\phi\circ\pi||_{\theta}\leq
D||\phi\circ\pi||_{\theta}e^{S_n(\psi)(\pi(y))}e^{-nP(\psi)}.
$$
In the case of expanding maps we have that $P(\psi_f)=0$ and $S_n(\psi_f)(\pi(y))\leq -n\log\lambda_f$, where $\lambda_f>1$ is the expansion constant for $f$, that is, $|f'(x)|\geq\lambda_f$ for all $x\in S^1$.
This proves that 
$$
\bigg|\int\phi d\mu_\psi^N-\int\phi d\mu_\psi\bigg|\leq C'||\phi\circ\pi||_{\theta}\tau^N
$$
For some $0<\tau<1$ (not necessarily the same $\tau$ as in Theorem 3.1).
Clearly $||\phi\circ\pi||_{\infty}\leq||\phi||_{\infty}$. Moreover

$$|\phi\circ\pi|_{\theta}=\sup_{x\neq y}\frac{|\phi(\pi(x)-\phi(\pi(y))|}{d_\theta(x,y)}\leq |\phi|_{Lip}\frac{d(\pi(x),\pi(y))}{d_\theta(x,y)}\leq |\pi|_\theta|\phi|_{Lip},$$
so that $||\phi\circ\pi||_\theta\leq \max\{1,|\pi|_{\theta}\}||\phi||_{Lip}.$ Thus, absorbing all constants into $C$, we have 
$$\bigg|\int\phi d\mu_\psi^n-\int\phi d\mu_\psi\bigg|\leq C||\phi||_{Lip}\tau^N,
$$
as desired.

\end{proof}
\begin{remark}
   One can prove effective equidistribution for general expanding repellers and Axiom A diffeomorphism using Markov partitions in the same way, but more care is needed to handle the difference between the measure $\nu_{\psi\circ\pi}^n$ and $\mu_{\psi\circ\pi}^n$. 
\end{remark}

\appendix
\section{Birkhoff Cones for Subshifts of Finite Type}
Rather than deducing our desired bounds on $||\mathcal{N}_{\psi_f}^n||_\theta$ as a consequence of quasi-compactness of the transfer operator (the standard approach of the Ruelle-Perron-Frobenius theorem), we use the technique of Birkhoff cones. The idea is to show that the transfer operator contracts a certain cone of Lipschitz functions with respect to a ``pseudo-metric" and to then establish the leading eigenfunction as a fixed point with respect to this pseudo-metric. The benefit to this approach is that we can establish explicit bounds on $||\mathcal{N}_{\psi_f}^n||_\theta$ which will be uniform in our set $\mathcal{W}$. Then one can actually deduce quasi-compactness as a consequence of this bound. This approach is standard for uniformly expanding maps; see \cite{Baladi} section  2.2. For subshifts of finite type, we follow \cite{Frédéric} closely, applied to the specific case of the full shift that we need, and for which certain technical difficulties vanish.
\begin{definition}
    A subset $\Lambda\subset \mathcal{B}/\{0\}$ of a Banach space $\mathcal{B}$ is called a cone if $\lambda\phi\in\Lambda$ for all $\phi\in\Lambda$ and all $\lambda>0$. The cone is called closed if $\Lambda\cup \{0\}$ is closed, and $\Lambda$ is called convex if $\psi_1+\psi_2\in\Lambda$ for every $\psi_1,\psi_2\in\Lambda$. A cone $\Lambda$ induces a partial order $\leq_\Lambda$ on $\mathcal{B}$ by defining
    $\psi\leq_\Lambda\phi \iff \phi-\psi\in\Lambda\cup\{0\}$.
\end{definition}
\begin{definition}
    For $\psi$ and $\phi$ in a cone $\Lambda$, define
    $$
    \alpha(\phi,\psi)=\sup\{\lambda>0 \hspace{1mm} | \hspace{1mm} \lambda\phi\leq_\Lambda\psi\},\hspace{1mm} \beta(\phi,\psi)=\inf\{\lambda>0\hspace{1mm} |\hspace{1mm} \psi\leq_\Lambda\lambda\phi \}.
    $$
    Then we define the Hilbert pseudo-metric $\Theta_\Lambda$ on $\lambda$ by
    $$
    \Theta_\Lambda(\phi,\psi)=\log\frac{\beta(\phi,\psi)}{\alpha(\phi,\psi)}.
    $$
\end{definition}
\begin{theorem}[Birkhoff's Inequality]
    Let $\Lambda$ be a convex cone in a Banach space $\mathcal{B}$. If $T:\mathcal{B}\rightarrow\mathcal{B}$ is a linear operator such that $T(\Lambda)\subset\Lambda$, then for each $\phi,\psi\in\Lambda$ we have
    $$
    \Theta_\Lambda(T\phi,T\psi)\leq \tanh\left({\frac{\diam_{\Theta_\Lambda}(T\Lambda)}{4}}\right)\Theta_\Lambda(\phi,\psi).
    $$
\end{theorem}
\begin{lemma}
    Let $\Lambda$ be a closed convex cone in a Banach space $\mathcal{B}$ endowed with two (not necessarily equivalent norms $||\cdot||_i$, $i=1,2,$ and assume that for all $\phi,\psi\in\mathcal{B}$
    $$
    -\phi\leq_\Lambda\psi\leq_\Lambda \phi\implies ||\psi||_i\leq||\phi||_i,\hspace{1mm} i=1,2.
    $$
    Then, for any $\phi,\psi\in\Lambda$ with $||\phi||_1=||\psi||_1$, we have
    $$
    ||\phi-\psi||_2\leq \left(e^{\Theta_\Lambda(\phi,\psi))}-1\right)||\phi||_2.
    $$
\end{lemma}

If $\Sigma^+$ is the one-sided full shift on $k$-symbols, and if $\mathcal{F}_\theta^+$ is the Banach space of Lipschitz continuous functions on $\Sigma^+$ with respect to the $d_\theta$-metric, then given any $L>0$, we have a cone in $\mathcal{F}_\theta^+$ given by
$$
\mathcal{C}_L=\{\phi\in\mathcal{F}_\theta^+ | \phi\geq 0, \phi\not\equiv 0, d_\theta(x,y)\leq\theta \implies \phi(x)\leq e^{Ld_\theta(x,y)}\phi(y)\}.
$$
In order to apply Birkhoff's inequality, we will need the following lemmas:
\begin{lemma}
    Fix $0<\xi<1$. Then for every $\phi,\psi\in\mathcal{C}_{\xi L}$ with $\phi,\psi>0$, we have 
    $$
    \Theta_{L}(\phi,\psi)\leq 2\log\left(\frac{1+\xi}{1-\xi}\right)+\log\sup_{x,y\in\Sigma^+}\left(\frac{\phi(x)\psi(y)}{\phi(y)\psi(x)}\right)
    $$
\end{lemma}
See \cite{Frédéric} Proposition 5.3 for the proof, which is unchanged in our setting.
\begin{lemma}
    Fix $\theta<\xi<1$. Then for every $L\geq\frac{\theta|\psi|_\theta}{\xi-\theta}$, we have $\mathcal{L}_\psi(\mathcal{C}_L)\subset\mathcal{C}_{\xi L}$, and we have 
    $$
    \diam_{\Theta_L}(\mathcal{L}_\psi(\mathcal{C}_L))\leq 2\log\left(\frac{1+\xi}{1-\xi}\right)+2\xi L.
    $$
\end{lemma}
\begin{proof}
    Let $\phi\in\mathcal{C}_L$, and let $x,y\in\Sigma^+$ be such that $d_\theta(x,y)\leq\theta$. We obtain
    $$
    \mathcal{L}_\psi\phi(x)=\sum_{i=1}^ke^{\psi(ix)}\phi(ix)\leq e^{\theta(|\psi|_\theta+L)d_\theta(x,y)}\sum_{i=1}^ke^{\psi(iy)}\phi(iy)
    =e^{\theta(|\psi|_\theta+L)d_\theta(x,y)}\mathcal{L}_\psi\phi(y).
    $$
    The condition $\theta|\psi|_\theta+L\theta\leq \xi L$ holds if and only if $L\geq\frac{\theta|\psi|_\theta}{\xi-\theta}$.
    Notice that $\phi\in\mathcal{C}_L$ implies that there exists at least one cylinder set $C_i=\{x\in\Sigma^+ | x_0=i\}$ such that $\phi|_{C_i}>0$. Thus
    $\mathcal{L}_\psi\phi(x)\geq e^{\psi(ix)}\phi(ix)>0$. We may therefore apply Lemma A.2 to functions $\mathcal{L}_\psi\phi_1$ and $\mathcal{L}_\psi\phi_2$:
    $$\Theta_L(\mathcal{L}_\psi\phi_1,\mathcal{L}_\psi\phi_2\leq 2\log\left(\frac{1+\xi}{1-\xi}\right) +
    \log\sup_{x,y\in\Sigma^+}\left(\frac{\mathcal{L}_\psi\phi_1(x)\mathcal{L}_\psi\phi_2(y)}{\mathcal{L}_\psi\phi_1(y)\mathcal{L}_\psi\phi_2(x)} \right)\leq$$
    $$
    2\log\left(\frac{1+\xi}{1-\xi}\right)+\log(e^{2\xi L})=
    2\log\left(\frac{1+\xi}{1-\xi}\right)+2\xi L.
    $$
\end{proof}

\noindent For $\phi\in\mathcal{F}^+$, define the seminorm
$$V(\phi):=\sup_{d_\theta(x,y)\leq\theta, x\neq y}\frac{|\phi(x)-\phi(y)|}{d_\theta(x,y)},$$ and set
$||\phi||_L:=\max (||\phi||_\infty, \frac{1}{2L}V(\phi))$.
The next lemma gives the essential properties of the norm $||\cdot||_L$:
\begin{lemma}

    The norm $||\cdot||_L$ is equivalent to $||\cdot||_\theta$, and for all $\phi,\psi\in\mathcal{F}^+$ we have that $-\phi_2\leq_{\mathcal{C}_L}\phi_1\leq_{\mathcal{C}_L} \phi_2$ implies that
    $||\phi_1||_L\leq ||\phi_2||_L$.
    
\end{lemma}
\begin{proof} 
    Given $\phi\in\mathcal{C}_L$, $\epsilon>0,$ and $x,y\in\Sigma^+$ such that $d_\theta(x,y)\leq \theta$, we have
    $$
    |\phi(x)-\phi(y)|=|e^{(\phi(x)+\epsilon)}-e^{(\phi(y)+\epsilon)}|\leq (||\phi||_\infty+\epsilon)\left|\frac{\phi(x)+\epsilon}{\phi(y)+\epsilon}\right|\leq (||\phi||_\infty+\epsilon)Ld_\theta(x,y).
    $$
    Letting $\epsilon\rightarrow 0$ and taking the supremum, we get the estimate $V(\phi)\leq ||\phi||_\infty$. This gives us that $||\phi||_L\leq C||\phi||\theta$ for some $C>0$. Likewise, it is easy to see that $|\phi|_\theta\leq 2||\phi||_L$, so we have that the norms are equivalent.

    Now suppose that $-\phi_2\leq_L\phi_1\leq_L\phi_2$. Then $\phi_2-\phi_1\geq 0$ and $\phi_2+\phi_1\geq 0.$ In other words, for every $x\in\Sigma^+$, $-\phi_2(x)\leq\phi_1(x)\leq\phi_2(x)$, which implies that
    $||\phi_1||_\infty\leq ||\phi_2||_\infty$. To prove that $||\phi_1||_L\leq||\phi_2||_L$, it suffices to prove that $V(\phi_1)\leq 2L||\phi_2||_\infty$. We have
    $$
    V(\phi_1)=V(\frac{\phi_1-\phi_2}{2}-\frac{\phi_1+\phi_2}{2})\leq \frac{1}{2}(V(\phi_2-\phi_1)+V(\phi_2+\phi_1))\leq $$
    $$
    \frac{L}{2}(||\phi_2-\phi_1||_\infty+||\phi_2+\phi_1||_\infty)\leq 2L||\phi_2||_\infty.
    $$
    
\end{proof}
\noindent Observe that for $\phi\in\mathcal{F}^+,$ $\phi\geq 0$ and $\alpha=\frac{|\phi|_\theta}{L>0}$, we have for all $x,y\in\Sigma^+$
    $$
    \frac{\phi(x)+\alpha}{\phi(y)+\alpha}=e^{\log(\phi(x)+\alpha)-\log(\phi(y)+\alpha)}\leq e^{\frac{|\phi|_\theta}{\alpha}d_\theta(x,y)}=e^{Ld_\theta(x,y)}.
    $$
    Hence $\phi+\alpha\in\mathcal{C}_L$.

\begin{proof}[Proof of Lemma 3.1]

Note that for every $f\in\mathcal{W}$, if we take $\theta>\frac{1}{\gamma}$, we have $|\pi_f|_\theta\leq 1$. Thus for uniformity in Lemma 3.1, we need to take 
$$
L\geq \frac{\theta M}{\xi-\theta},
$$
where $M:=\max\{|\log(f')|_{C^1}\hspace{1mm} |\hspace{1mm} f\in\mathcal{W}\}$. Let $h_f$ be such that $\mathcal{L}_{\psi_f}h_f=e^{P(\psi_f)}h_f$, and let $\nu_f$ be the corresponding measure such that $\mathcal{L}_{\psi_f}^*\nu_f= e^{P(\psi_f)}\nu_f.$ Note that for all $f\in\mathcal{W}$, $P(\psi_f)=0$
It can be shown that $h_f\in\mathcal{C}_L$ for $L$ taken as above. Hence for any $x,y\in\Sigma^+$ we have
$$
h_f(x)=\mathcal{L}_{\psi_f}h_f(x)=\sum_{i=1}^ke^{\psi_f(ix)}h_f(ix)\leq e^{\theta(L+|\psi_f|_\theta)d_\theta(x,y)}\sum_{i=1}^ke^{\psi_f(iy)}h_f(iy)\leq e^{\theta(L+M)d_\theta(x,y)}h_f(y),
$$
and hence $|\log(h_f)|_\theta\leq \theta(L+M)$, and so 
$|\overline{\psi}_f|_\theta\leq M+\theta(L+M)+L+M=(2+\theta)M+(1+\theta)L$. Thus for the normalized operators $\mathcal{L}_{\overline{\psi}_f}$ we take 
$$
L_0\geq \frac{\theta((2+\theta)M+(1+\theta)L)}{\xi-\theta}.
$$
Moreover, one can show that 
$-\phi\leq_{\mathcal{C}_L}\psi\leq_{\mathcal{C}_L} \phi $
implies that 
$$
\int\psi d\mu_f\leq \int\phi d\mu_f,
$$
where $\mu_f$ is the equilibrium state corresponding to $\overline{\psi}_f$.

Observe that for every $n\in\mathbb{N}$,
$$
\int\mathcal{L}_{\overline{\psi}_f}\phi d\mu_f=\int\phi d\mu_f.
$$
Therefore we may apply Lemma A.1 with $||\cdot||_1=||\cdot||_{L^1}$, $||\cdot||_2=||\cdot||_L$, $\phi=\mathcal{L}^n_{\overline{\psi}_f}\phi$, and $\psi=\int\phi d\mu_f=\mathcal{L}^n_{\overline{\psi}_f}(\int\phi d\mu_f)$:
$$
||\mathcal{L}^n_{\overline{\psi}_f}\phi-\int\phi d\mu_f||_L\leq \left(e^{\Theta_L(\mathcal{L}^n_{\overline{\psi}_f}\phi,\mathcal{L}^n_{\overline{\psi}_f}(\int\phi d\mu_f))}-1\right)||\int\phi d\mu_f||_L\leq
\left(e^{\Theta_L(\mathcal{L}^n_{\overline{\psi}_f}\phi,\mathcal{L}^n_{\overline{\psi}_f}(\int\phi d\mu_f))}-1\right)||\phi||_L
$$
Let $\Delta=\diam_{\Theta_L}(\mathcal{L}_{\overline{\psi}_f}(\mathcal{C}_L))$, and observe that Birkhoff's inequality implies that for $\phi\in\mathcal{C}_L$
$$
\Theta_L(\mathcal{L}^n_{\overline{\psi}_f}\phi,\mathcal{L}^n_{\overline{\psi}_f}(\int\phi d\mu_f))\leq
\left(\tanh(\frac{\Delta}{4})\right)^{n-1}\Delta\leq 
\Delta\tau_{\mathcal{W}}^{n-1},
$$
for uniform $\tau_{\mathcal{W}}$.Therefore for $\phi\in\mathcal{C}_L$, we have
$$
||\mathcal{L}^n_{\overline{\psi}_f}\phi-\int\phi d\mu_f||_L\leq
\left(\sum_{j=1}^\infty\frac{(\Delta\tau_{\mathcal{W}}^{n-1})^j}{j!}\right)||\phi||_L=C_{\mathcal{W}}\tau_{\mathcal{W}}^n||\phi||_L,
$$
where $C_{\mathcal{W}}$ is uniform. Since the norm's are equivalent, we may replace $||\cdot||_L$ by $||\cdot||_\theta$. It remains to extend this bound to all $\phi\in\mathcal{F}_\theta^+$. If $\phi\geq 0$, then $\phi+\frac{|\phi|_\theta}{L}\in\mathcal{C}_L$, so 
$$
||\mathcal{L}^n_{\overline{\psi}_f}\phi-\int\phi d\mu_f||_\theta=||\mathcal{L}^n_{\overline{\psi}_f}(\phi+\frac{|\phi|_\theta}{L})-\int(\phi+\frac{|\phi|_\theta}{L}) d\mu_f||_\theta\leq C_{\mathcal{W}}\tau_{\mathcal{W}}^n(||\phi||_\theta+\frac{|\phi|_\theta}{L})\leq 
C_{\mathcal{W}}\tau_{\mathcal{W}}^n||\phi||_\theta
$$
for a different (but still uniform) $C_{\mathcal{W}}$.
For general real-valued $\phi\in\mathcal{F}_\theta^+$, we decompose $\phi$ as $\phi=\phi^+-\phi^-$, where $\phi^+,\phi^-\geq 0$. Then
$$
||\mathcal{L}^n_{\overline{\psi}_f}\phi-\int\phi d\mu_f||_\theta\leq
||\mathcal{L}^n_{\overline{\psi}_f}\phi^+-\int\phi^+ d\mu_f||_\theta+||\mathcal{L}^n_{\overline{\psi}_f}\phi^-\int\phi^- d\mu_f||_\theta\leq
$$
$$
C_{\mathcal{W}}\tau_{\mathcal{W}}^n(||\phi^+||_\theta+||\phi^-||_\theta)\leq
2C_{\mathcal{W}}\tau_{\mathcal{W}}^n||\phi||_\theta.
$$
The case of complex-valued $\phi$ is handled similarly.
\end{proof}
\printbibliography

\end{document}